\documentclass[fleqn,a4paper,12 pt, twoside]{article} 
\usepackage[
	top = .75 in, 
	bottom = 1 in,
	left = .75 in, 
	right = .75 in,
	includehead]{geometry}

\usepackage{xcolor}

\usepackage{scrextend}
\changefontsizes{14pt}

\usepackage[utf8x]{inputenc}
\usepackage{amsfonts,amssymb,amsmath}
\usepackage{amsrefs}
\usepackage{url}

\BibSpec{article}{%
    +{}  {\PrintAuthors}                {author}
    +{,} { \textit}                     {title}
    +{.} { }                            {part}
    +{:} { \textit}                     {subtitle}
    +{,} { \PrintContributions}         {contribution}
    +{.} { \PrintPartials}              {partial}
    +{,} { }                            {journal}
    +{}  { \textbf}                     {volume}
    +{}  { \PrintDatePV}                {date}
    +{,} { \issuetext}                  {number}
    +{,} { \eprintpages}                {pages}
    +{,} { }                            {status}
    +{,} { \url}                        {url}    
    +{,} { \PrintDOI}                   {doi}
    +{,} { available at \eprint}        {eprint}
    +{}  { \parenthesize}               {language}
    +{}  { \PrintTranslation}           {translation}
    +{;} { \PrintReprint}               {reprint}
    +{.} { }                            {note}
    +{.} {}                             {transition}
    +{}  {\SentenceSpace \PrintReviews} {review}
}

\usepackage{graphicx}
\usepackage[
	british]{babel}
\usepackage{caption}
\usepackage{mathdots}
\usepackage{mathtools} 
\usepackage{amsmath}
\usepackage{amsfonts}
\usepackage{amssymb}

\usepackage{pgf,tikz}
\usepackage{tikz-cd}
\usetikzlibrary{calc}
\usetikzlibrary{matrix,arrows}
\usepackage{stackrel}
\usepackage[shortlabels]{enumitem} 
\usepackage{stmaryrd} 
\usepackage{setspace} 
\spacing{1.15}
\usepackage{etoolbox}
\usetikzlibrary{trees}
\usepackage{enumitem}
\setlist[enumerate]{label= (\arabic*)}

\DeclareFontFamily{U}{wncy}{}
    \DeclareFontShape{U}{wncy}{m}{n}{<->wncyr10}{}
    \DeclareSymbolFont{mcy}{U}{wncy}{m}{n}
    \DeclareMathSymbol{\Sh}{\mathord}{mcy}{"58}

\patchcmd{\section}{\normalfont}{\normalfont\large}{}{}
\usepackage{fourier}
\usepackage[T1]{fontenc}
 \usepackage{wrapfig, framed, caption}
 \usepackage{fancyhdr}
\usepackage{letltxmacro}

\AtBeginDocument{%
    \LetLtxMacro\refa{\ref}%
    \DeclareRobustCommand{\ref}[2][]{(\refa#1{#2})}%
}

\DeclareTextFontCommand{\new}{\color{black}\em}

\usepackage{hyperref}
\hypersetup{
colorlinks,linkcolor={newcol},citecolor={newcol},urlcolor={newcol}}  

\usepackage{cutwin}

\tikzcdset{arrow style=tikz, diagrams={>=stealth}}

\makeatletter
\AtBeginDocument{%
  \let\[\@undefined
  
\DeclareRobustCommand{\[}{\begin{equation}}%
  \let\]\@undefined
  
\DeclareRobustCommand{\]}{\end{equation}}%
}
\makeatother 
\mathtoolsset{showonlyrefs,showmanualtags}

\usepackage{stmaryrd}
\usepackage{amsthm}
\usepackage{thmtools}
\newtheoremstyle{mytheorem}
  {\topsep}   
  {\topsep}   
  {\itshape}  
  {0pt}       
  {\bfseries\color{newcol}} 
  {\color{newcol}{.}}         
  {5pt plus 1pt minus 1pt} 
  {}          
\theoremstyle{mytheorem}
\newtheorem{theorem}{Theorem}[section]
\newtheorem{cor}{Corollary}[section]
\newtheorem{lemma}{Lemma}[section]

\newcommand{\antishriek}{\text{\raisebox{\depth}{\textexclamdown}}}

\newcommand{\hAss}{\mathsf{Ass}_\infty}
\newcommand{\hLie}{\mathsf{Lie}_\infty}
\newcommand{\hPois}{\mathsf{Pois}_\infty}
\newcommand{\Lie}{\mathsf{Lie}}
\newcommand{\Ass}{\mathsf{Ass}}
\newcommand{\Com}{\mathsf{Com}}
\newcommand{\Pois}{\mathsf{Pois}}
\newcommand{\gr}{\mathsf{gr}}

\newcommand{\g}{\mathfrak{g}}

\DeclareFontFamily{U}{wncy}{}
    \DeclareFontShape{U}{wncy}{m}{n}{<->wncyr10}{}
    \DeclareSymbolFont{mcy}{U}{wncy}{m}{n}
    \DeclareMathSymbol{\Sh}{\mathord}{mcy}{"78} 
\newenvironment{titemize}{
\begin{itemize}
  \setlength{\itemsep}{0pt}
  \setlength{\parskip}{0pt}
}{\end{itemize}}

\theoremstyle{mytheorem}
\newtheorem*{theorem*}{Theorem}
\newtheorem*{lemma*}{Lemma}
\newtheorem*{corollary*}{Corollary}
\newtheorem*{conjecture*}{Conjecture}
\newcommand{\imor}{\interleave\kern-.45em\longrightarrow}

\newcommand{\Der}{\operatorname{Der}}

\newenvironment{tenumerate}{
\begin{enumerate}
  \setlength{\itemsep}{0pt}
  \setlength{\parskip}{0pt}
}{\end{enumerate}}

\newcommand{\Mod}{\mathsf{Mod}}

\newcommand{\Alg}{\mathsf{Alg}}
\newcommand{\Cog}{\mathsf{Cog}}

\newcommand{\?}{\,?\,}

\newcommand{\NN}{\mathbb N}

\newcommand{\kk}{\Bbbk}

\newcommand{\Tor}{\operatorname{Tor}}

\newcommand{\chLie}{\mathsf{ho}(\mathsf{dg Lie})}
\newcommand{\chAss}{\mathsf{ho}(\mathsf{dg Alg})}

\newcommand{\Ch}{\mathsf{Ch}}

\usepackage[new]{old-arrows} 

\newcommand{\ttt}{\texttt{-}}
\newcommand{\TwAs}{\mathsf{Tw}\ttt\mathsf{As}}

\DeclareTextFontCommand{\new}{\color{newcol}\bf\em} 

\newcommand\claim[2][.8]{%
  \begin{minipage}{#1\displaywidth}%
  \itshape
  #2
  \end{minipage}%
}

\newcommand{\Addresses}{{
  \bigskip
  \footnotesize
  \textsc{School of Mathematics, Trinity College Dublin}, College Green, Dublin 2, Ireland, D02 PN40\par\nopagebreak
  \textit{E-mail address:} \texttt{pedro@math.tcd.ie}
  \bigskip
  
\textsc{
Department of Mathematics, National Research University Higher School of Economics}, 
20 Myasnitskaya street, Moscow, Russia, 101000
and \textsc{Institute for Theoretical and Experimental Physics}, Bolshaya Cheremushkinskaya 25, Moscow, Russia, 117259\par\nopagebreak
  \textit{E-mail address:}
  \texttt{akhoroshkin@hse.ru}  }}
  
\usepackage{titletoc}

\titlecontents{chapter}
[0.2em] %
{\bigskip}
{\makebox[2em][r]{\thecontentslabel.}\hspace{0.333em}}
{\hspace*{-2em}}
{\hfill\contentspage}[\smallskip]

\titlecontents{section}
[2.7em]
{\small}
{\thecontentslabel.\hspace{3pt}}
{}
{\enspace\titlerule*[0.5pc]{.}\contentspage}
\titlecontents*{subsection}
[3.58em]
{\footnotesize}
{\thecontentslabel. \hspace{3pt}}
{}
{}
[ --- \ ]
[]
\definecolor{seccol}{rgb}{0.00, 0.20, 0.40}
\definecolor{defcol}{rgb}{0.00, 0.20, 0.40}
\definecolor{ssecol}{rgb}{0.00, 0.20, 0.40}
\definecolor{newcol}{rgb}{0.00, 0.20, 0.40} 
\definecolor{chacol}{rgb}{0.00, 0.20, 0.40}

\usepackage{sectsty}
\chapterfont{\color{chacol}}  
\sectionfont{\color{seccol}}  
\subsectionfont{\color{ssecol}}  

\makeindex

\title{\textsc{Derived Poincaré--Birkhoff--Witt 
theorems}\\ \tiny\textcolor{white}{space}\\
\small\textit{with an appendix by Vladimir Dotsenko}
}
\author{
\textsc{Anton Khoroshkin and Pedro Tamaroff}
}
\date{}
\begin{document}
\maketitle

\begin{abstract}

We propose a new general formalism that 
allows us to study Poincaré--Birkhoff--Witt type phenomena
for universal enveloping algebras in the 
differential graded context. Using it,  
we prove a homotopy invariant version 
of the classical Poincaré--Birkhoff--Witt 
theorem for universal envelopes
of Lie algebras. In particular, our results imply 
that all the previously known constructions of universal
envelopes of $L_\infty$-algebras (due to Baranovsky, Lada and Markl, and Moreno-Fernández) represent
the same object of the homotopy category of 
differential graded associative algebras. We also extend Quillen's classical quasi-isomorphism $\mathcal C \longrightarrow BU$
from differential graded Lie algebras to $L_\infty$-algebras; this 
confirms a conjecture of Moreno-Fernández. 
\end{abstract}


\tableofcontents

\pagebreak

\section*{Introduction}\label{sec:intro}
\addcontentsline{toc}{section}{\nameref{sec:intro}}

It is a classical result going back
to Poincar\'e, Birkhoff and Witt that,
over a field of characteristic
zero, the universal enveloping algebra
$U(\g)$ of a Lie algebra $\g$ is 
isomorphic, as a vector space, to the 
symmetric algebra $S(\g)$ on $\g$. One 
can, in fact, extend without changes
the definition of 
the functor $U$ to the category of dg Lie
algebras, and a consequence of the PBW
theorem is that for any dg Lie algebra we 
have a natural isomorphism $UH(\g)
\longrightarrow HU(\g)$. This result
implies that the underlying homology group of the universal envelope of a dg Lie algebra
does not depend on the Lie algebra
structure of $\g$ but only on the homology group $H(\g)$, and that 
this functor descends to the homotopy 
category of dg Lie algebras.
This category 
can be modeled through minimal
$L_\infty$-algebras and their morphisms
up to $L_\infty$-quasi-isomorphism and,
similarly, the homotopy category of dga
algebras can be modeled through minimal
$A_\infty$-algebras and their morphisms
up to $A_\infty$-quasi-isomorphism, so
it is reasonable to consider the problem
of finding a functor at the level of
$L_\infty$-algebras representing $U$ on
the homotopy category. 

This program
has been carried out by
V. Baranosvky~\cite{Bar} and later by J. Moreno-Fernandez~\cite{Jose}, 
and their results imply that to every
minimal $L_\infty$-algebra one can assign 
a ``universal enveloping'' minimal 
$A_\infty$-algebra $\Upsilon(\g)$ that enjoys many 
properties similar to those of
the classical 
universal enveloping algebra functor.
\begin{titemize}
\item its underlying 
vector space is the symmetric algebra
$S(\g)$, independently of the $L_\infty$-structure of $\g$,
\item there exists a Quillen type 
$A_\infty$-quasi-isomorphism $\Omega\mathcal{C}(\g)\longrightarrow \Upsilon(\g)$ (so the algebra $\Upsilon(\g)$ does have the correct homotopy type),
\item the canonical inclusion $\g\longrightarrow \Upsilon(\g)$ is a strict $L_\infty$-morphism: the antisymmetrized higher multiplication maps on $\Upsilon(\g)$ restrict to the higher brackets of $\g$. 
\end{titemize}

One of the
goals of this paper is to explain, 
using the methods of
V. Dotsenko and second author, introduced in~\cite{PBW} and suitably extended
to the dg setting, that the results 
of~\cite{Bar} and~\cite{Jose} both
follow from a derived version
of the classical Poincar\'e-Birkhoff-Witt theorem for Lie 
algebras. To accomplish this, we study, through the theory of operads, the Lada--Markl functor that assigns to an $A_\infty$-algebra the $L_\infty$-algebra obtained by antisymmetrising all the product operations. As in the non-dg case, this functor has a left adjoint, and we prove it satisfies
a \emph{derived Poincar\'e--Birkhoff--Witt theorem}. Our main
result implies the following; here $U(\g)$
denotes the left adjoint above, while $UH(\g)$ 
is the universal envelope
of the Lie algebra $H(\g)$: 

\begin{theorem*} For every minimal 
$L_\infty$-algebra $\g$: 
\begin{tenumerate}
\item there is an
isomorphism between the symmetric algebra
$S(\g)$ and the homology $HU(\g)$ of the $A_\infty$-universal envelope $U(\mathfrak{g})$. This isomorphism is natural with respect to
strict $L_\infty$-morphisms.
\item $S(\g)$ can be endowed with a minimal $A_\infty$-algebra structure $A_\infty$-quasi-isomorphic to $U(\g)$ so that the inclusion $\g \longrightarrow S(\g)$
is a strict map of $L_\infty$-algebras.\qed 
\end{tenumerate}
\end{theorem*}

\noindent From this theorem, we deduce the three
proposed models for universal envelopes
of $L_\infty$-algebras are, up to homotopy, one and the same. This means,
in particular, that the two existing
 tentative
models for universal envelopes are
in fact models. This follows from the
following:

\begin{corollary*}
Let $\g$ be a minimal $L_\infty$-algebra.
Every minimal $A_\infty$-algebra structure
on $S(\g)$ for which the restriction
of the antisymmetrized
structure operations coincide with the 
operations of $\g$ is 
$A_\infty$-quasi-isomorphic to
$U(\g)$. Thus, the models
of Baranovsky
and Moreno-Fernandez are $A_\infty$-isomorphic, and $A_\infty$-quasi-isomorphic to $U(\g)$.\qed
\end{corollary*}
\noindent For a final application,
consider an
$L_\infty$-algebra $\g$. Then $ 
\mathcal C(\g)$ is a commutative dgc coalgebra, 
where $\mathcal C$ is the
Quillen construction on $\g$. Since $U(\g)$
is an $A_\infty$-algebra, it is natural
to compare $\mathcal C(\g)$ to $BU(\g)$,
which is a
(non-commutative) dgc coalgebra. This
result was conjectured in~\cite{Jose}.

\begin{corollary*}
There is an acyclic cofibration $\mathcal C(\g) \longrightarrow BU(\g)$ of dgc coalgebras natural with respect to strict $L_\infty$-morphisms. Moreover,
for any minimal $L_\infty$-algebra $\g$, there are acyclic cofibrations
$\mathcal{C}(\g) \longrightarrow BS(\g)$. In particular, $BU(\g)$ always admits a commutative
model.  \qed
\end{corollary*}

\textbf{\color{newcol}Structure.} 
The paper is organised
as follows. In Section~\ref{Sec1} we prove our
main theorem relating almost-free and derived PBW morphisms between
dg operads. We then use
it in Section~\ref{Sec2} to show that
the map from the homotopy
Lie operad to the homotopy associative operad is derived PBW and
deduce from this several results on universal envelopes of $L_\infty$-algebras, which recover results from Baranovsky and Moreno-Fernandez, and extend the Quillen quasi-isomorphism $\mathcal C\longrightarrow BU$ to $L_\infty$-algebras. We also apply 
our main theorem to show 
\emph{associative} 
universal envelopes 
satisfy the derived PBW property as 
soon as they are PBW.
The appendix, written by V.~Dotsenko, contains a general result on models of operads obtained by homological perturbation which we use in the particular case of the homotopy associative and the homotopy Poisson operad.

\bigskip
  
\textbf{\color{newcol}Notation and conventions.} 
We work over a field of
characteristic zero, which we write $\kk$. We assume the reader is familiar with the theory of
algebraic operads,
as presented, for example, in~\cite{LV}, with the elements of model theory, as presented, for example, in~\cite{Hov}, and
with the basic tools of homological algebra. 
Whenever a new definition is provided, it will 
appear in \new{boldfaced italics}. 

\bigskip

\textbf{\color{newcol}Acknowledgements.}
This paper saw great progress while the second author
was visiting the 
Higher School of Economics in Moscow.
We thank the Higher School of
Economics for their hospitality and
wonderful working conditions.
We thank
Vladimir Dotsenko for his constant
support and valuable advice during
the preparation of these notes, for
his insight into PBW theorems for
operads that motivated this sequel
to the joint work in~\cite{PBW}, and
for his careful reading of the manuscript.

We thank Alexander Efimov for useful
discussions that encouraged us
to write Section 2.3 and Ricardo Campos,
Daniel Robert-Nicoud and Luis Scoccola
for their useful comments and suggestions.
We also thank
Ben Knudsen for answering some questions
about his work~\cite{Knud} on enveloping $E_n$-algebras of
spectral Lie algebras, and Guillermo
Tochi for pointing us to this paper in the
first place. 

The research of A. Kh. was carried out within the HSE University Basic Research Program
and supported in part by the Russian Academic Excellence Project '5-100' and in part by the Simons Foundation.
\vspace*{\fill}
\pagebreak
\section{The derived PBW property}\label{Sec1}

For convenience, we remind
the reader of the language of~\cite{PBW}. Let
us fix a symmetric monoidal category $\mathsf{C}$
that can be either of $\mathsf{Vect}_\kk$, the 
category of graded vector spaces, $\mathsf{Ch}_\kk$, the category of complexes over $\kk$, ${}_{\Sigma}\mathsf{Mod}$, the category of
$\Sigma$-modules under the Cauchy tensor product
or $\mathsf{Ch}_\Sigma$, the category of dg
$\Sigma$-modules under the same product. 
We say a map of operads $f: P 
\longrightarrow Q$ over $\mathsf{C}$
satisfies the \new{Poincar\'e-Birkhoff-Witt
property} if there is 
an endofunctor
$T : \mathsf{C}\longrightarrow\mathsf{C}$ so that universal enveloping
algebra functor $f_!$ is naturally
isomorphic to $T$ with
respect to $P$-algebra maps. In other words, we demand that $f_!$ depend only on the underlying object of that algebra in $\mathsf C$, but not on the operations of that particular $P$-algebra. With this at hand, the main result of~\cite{PBW} is the following (we refer the reader to~\cite{Fresse} for details on operads
and their modules):

\begin{theorem*} The map $f:P\longrightarrow Q$ satisfies the
Poincar\'e-Birkhoff-Witt property if
and only if $Q$ is a free right 
$P$-module. Moreover, in this case,
if $T$ is a basis for $Q$ as a right $P$-module, then $f_!$ is
isomorphic to $T$, naturally with respect to $P$-algebra maps.  \qed
\end{theorem*}

The new formalism we propose for dg
operads is as follows. Fix a dg operad $P$. We say a right dg 
$P$-module $X$ is \new{almost-free} if it admits
a bounded below exhaustive filtration so that its associated graded 
module is chain equivalent to a free right $P$-module 
on a basis of cycles. A morphism of dg operads
$f: P\longrightarrow Q$ is \new{almost-free} if
$Q$ is an almost-free right $P$-module. Note that
for every dg module $X$ we have a corresponding
right $HP$-module $HX$, where $H$ is the
homology functor. It follows that we have both
a left adjoint $f_!$ to the restriction
functor $f^* : {}_Q\Alg\longrightarrow {}_P\Alg$ and a left adjoint $(Hf)_!$ to the
restriction functor $(Hf)^* : {}_{HQ}\Alg\longrightarrow {}_{HP}\Alg$. In this way,
we obtain two functors 
\[H\circ f_! : {}_{P}\Alg\longrightarrow {}_{HQ}\Alg,\quad (Hf)_!\circ H : {}_{P}\Alg\longrightarrow {}_{HQ}\Alg\]
and a natural transformation $F : (Hf)_!\circ H \longrightarrow H\circ f_!$. 
The map $f$ is \new{derived Poincaré--Birkhoff--Witt} if $F$ is
a natural isomorphism.  Our main result is the following:

\begin{theorem}
Every morphism that is almost-free is derived PBW. 
Moreover, if $T$ is a basis of cycles for
such a morphism, the homology
$Hf_!$ of
its universal envelope is naturally isomorphic
to $TH$ as a functor of algebras on its 
domain to algebras over the homology of its
codomain. 
\end{theorem}

Thus, in the same way
that a classical PBW theorem gives an
amenable description of universal 
enveloping algebra dependent only on 
the underlying object of the input,
a derived PBW
theorem gives us an amenable
description of the homology of the
universal 
enveloping algebra dependent only on the homology 
of the input.

\begin{proof} We split our proof into three 
steps. Our main tool is a classical spectral
sequence argument.

\bigskip

\textbf{\color{newcol}Step 1.} Suppose that $Q$ is in fact $P$-free
on a basis of cycles, so that $Q = 
T\circ P$ with $dT=0$. Since $T$ has trivial differential,
the Künneth theorem for the circle product
gives a natural isomorphism $HQ \longrightarrow
T\circ HP$, and shows that $HQ$ is a free right $HP$-module. Moreover, for
every left $P$-module $X$, we have natural
isomorphisms $H(f_!(X))= H(Q\circ_P X)\longrightarrow H(T\circ X) \longrightarrow T \circ HX$. We also have natural isomorphisms
$T\circ HX \longrightarrow T\circ HP \circ_{HP} X \longrightarrow HQ\circ_{HP} HX = (Hf)_!(HX)$,
which gives what we wanted: the natural map
$F_X: (Hf)_!(HX) \longrightarrow H(f_!(X))$ is
a natural isomorphism, so that $f$ is derived
PBW in this case.

\smallskip

\textbf{\color{newcol}Step 2.} Let us consider now the situation where we
have an almost-free filtration $F$ on
$Q$, and consider the induced filtration on
$f_!(X)$. Linearity on the left of the composite product gives us
that $\gr_F(f_!(X))$ is of the form $T\circ X$ where $T$ has trivial differential,
so that the domain of the $E^1$-page of the maps of spectral sequences
converging to $F_X: (Hf)_!(H(X))\longrightarrow H(f_!(X))$ looks like $T\circ HP$.
The arguments above now show that the induced
morphism at the $E^1$-page  is
an isomorphism.

\smallskip

\textbf{\color{newcol}Step 3.} To conclude, let us suppose that we have
a filtration on $Q$ such that $\gr(Q)$ is
chain equivalent to a free right 
$P$-module $Q'$. Arguing as before, we have a map of spectral sequences converging to $F_X : (Hf_!)(HX)\longrightarrow H(f_!(X))  $
with $E^0$ equal to $(\gr\, Q)\circ_P X$, and a
map $(\gr\, Q)\circ_P T \longrightarrow Q'\circ_P X= T\circ X$. Since the circle
product is left linear, this map is still
a chain equivalence, and thus induces an
isomorphism on homology. 

\smallskip

The last claim is already
part of the content of the main result in~\cite{PBW}. This 
concludes the proof of the theorem.
\end{proof}

We point out that the work of V. Hinich~\cite{Hinich}*{Section 4.6.3} 
proves that universal envelopes preserve
acyclic cofibrations between algebras:
a derived PBW theorem extends this to arbitrary
weak equivalences.

\begin{cor}
If $f:P\longrightarrow Q$ is derived
PBW, then $f_!$ preserves weak equivalences. \qed
\end{cor}

\noindent It is worth pointing out that the PBW theorem
of V. Dotsenko and the second author in~\cite{PBW}
already implies the following classical result,
since there we show that the map of operads
$Hf : \Lie \longrightarrow \Ass$ is free and
thus PBW in the classical sense.

\begin{corollary*} Let $\g$ be a dg Lie
algebra. Then the natural map $UH(\g) \longrightarrow HU(\g)$ is an isomorphism of algebras, so
that a map of dg Lie algebras is
a quasi-isomorphism if and only if the map
on universal envelopes is one.\qed 
\end{corollary*}

\textbf{\color{newcol}Relation to a theorem of Adams.} It is
interesting to point out that in~\cite{Adams} (see also~\cite{Guide}*{Chapter 9}), the author shows
that if $f :\Lambda\longrightarrow \Gamma$ is
an inclusion  of Hopf algebras and $\Lambda$
is central in $\Gamma$, then $B\Gamma$ admits
a filtration whose graded coalgebra is chain
equivalent to $(B\Lambda\otimes B\Omega, d\otimes 1)$, so that $Bf$ is almost-free
according to our definition. In fact,
this result was the main inspiration for our
definition of almost-free morphisms. We remark that
in the context of differential graded homological
algebra, the objects which posses useful homological properties and replace, in a way,
the free objects of the classical  theory, are
sometimes called ``semi-free modules''. We chose
not to use this terminology to avoid any kind of
confusion. It is useful to 
note that one can just assume that $\Gamma$
is $\Lambda$-free along with the fact that
$\Lambda$ is central to deduce this. Hence,
Adams' result states precisely that, under this
last extra hypothesis, $B$ preserves almost-free maps. Having appropriate hypotheses and a similar result for $\Omega$, one would obtain 
results relating classical PBW maps of operads,
as defined in~\cite{PBW}, to derived PBW maps
between cofibrant replacements. We intend to 
pursue these ideas in the future.

\section{Applications}\label{Sec2}
In this section we prove the derived version of the classical PBW theorem,
which we then use to deduce  
results of V. Baranosvky and J. 
M. Moreno-Fernández,  both who constructed
universal envelopes for $L_\infty$-algebras, and
answer in the positive a 
conjecture of Moreno-Fernández regarding
the universal envelope construction
considered by Lada--Markl in~\cite{Lada}.
Observe that the motto of~\cite{PBW} that ``the universal 
envelope of $\g$ is independent 
of the Lie algebra structure of 
$\g$'' is now replaced by ``the 
homology of the universal 
envelope of $\g$ is independent 
of the homotopy Lie algebra 
structure of $\g$''. 
Since all
three constructions have their
particular intricacies, let us 
begin by recalling 
the essential definitions of the
respective universal envelopes
for $L_\infty$-algebras.

\bigskip

\textbf{\color{newcol}The universal envelope as a left adjoint.} Let us
recall the work of Lada--Markl~\cite{Lada} that
generalizes the well-known fact the
antisymmetrization of the product of
an associative algebra yields a 
Lie algebra. In~\cite{Lada}, the authors show that if $A$ is an 
$A_\infty$-algebra with higher
products $(m_1,m_2,m_3,\ldots)$ and
if we set, for each $n\in \mathbb N$,
\[\label{eq:ant}\tag{1} l_n = \sum_{\sigma\in S_n} (-1)^\sigma m_n\cdot \sigma \]
these maps define on $A$ an 
$L_\infty$-algebra structure. This
implies there is a map of operads
$f:\hLie\longrightarrow 
\hAss$ defined by (\ref{eq:ant}). As in the classical
case, it makes sense to define the
universal envelope of an $L_\infty$-algebra through the left adjoint $f_!$ of the map that assigns an $A_\infty$-algebra $A$ to the corresponding
$L_\infty$-algebra $f^*(A)$, which we write $A^\circ$, following
the prescription above: this
is the quotient of the free $A_\infty$-algebra on $\g$ by the relations imposed by the equation~(\ref{eq:ant}).  We will
write $U(\g)$ for $f_!(\g)$ and call it \new{the universal enveloping algebra of $\g$}; since this
universal algebra is given unequivocally by the same formalism
that defines the classical universal
envelope of dg Lie algebras,
we refrain from giving it any other
name. Note that by construction
there is a unit map $\g\longrightarrow U(\g)^\circ$ which is a strict 
morphism of $L_\infty$-algebras.

\bigskip

\textbf{\color{newcol}The construction of Baranosvky.} Let us recall that
if $\g$ is a dg Lie algebra, there
is a natural quasi-isomorphism of
dga algebras $q : \Omega \mathcal{C}(\g)\longrightarrow U(\g)$ which
assigns a generator of the Chevalley--Eilenberg complex $\mathcal{C}(\g)$ to its antisymmetrization in $U(\g)$. 
Baranosvky~\cite{Bar} shows 
that there is a contraction of
complexes
$\Omega \mathcal{C}(\g')\longrightarrow S(\g')$ where
$\g'$ is the abelian algebra associated to $\g$.
To define $U(\g)$ for the more
general class of $L_\infty$-algebras,
Baranovsky resorts to a perturbative
method, as follows. For such an algebra $\g$, the key ingredients of his construction are:
\begin{titemize}
\item the contraction of complexes $\Omega \mathcal{C}(\g')\longrightarrow S(\g')$,
\item the resulting contraction of dgc coalgebras
$B\Omega \mathcal{C}(\g')\longrightarrow BS(\g')$.
\item the fact the differential of $B\Omega \mathcal{C}(\g)$ is a perturbation of the differential of
$B\Omega \mathcal{C}(\g')$. 
\end{titemize}
This two facts imply, together with the homological perturbation lemma, that
there is on $BS(\g')$ a dgc coalgebra
structure that is quasi-isomorphic
to $B\Omega \mathcal{C}(\g)$. In
other words, there is on $S(\g')$
an $A_\infty$-algebra structure
that is $A_\infty$-quasi-isomorphic
to $\Omega \mathcal{C}(\g)$. Moreover, Baranosvky shows the
PBW inclusion $\g\longrightarrow S(\g)^\circ$ is a strict map of $L_\infty$-algebras. We 
call this the \new{Baranosvky universal enveloping algebra of $\g$}. Note that $\Omega \mathcal{C}(\g)$ has the same homotopy type
as $U(\g)$, so it makes sense to focus on this
object to elucidate a universal enveloping algebra,
since $\Omega \mathcal{C}(\g)$ exists for \emph{any} $L_\infty$-algebra.

\bigskip

\textbf{\color{newcol}The construction of Moreno-Fern\'andez.}
The approach of Moreno-Fern\'andez~\cite{Jose}
is slightly different from that of
Baranovsky, but is also perturbative
in nature. Let us take a dg Lie
algebra $\g$, and assume we have
a contraction of complexes from
$\g$ onto $H(\g)$. The homotopy
transfer theorem then guarantees
there is an $L_\infty$-structure on
$H(\g)$ which makes $H(\g)$ an $L_\infty$-algebra $L_\infty$-quasi-isomorphic to $\g$. The author then
shows there there is an explicit contraction of complexes $U(\g)\longrightarrow S(H(\g))$ which,
by the homotopy transfer theorem,
gives us an $A_\infty$-algebra structure on $S(H(\g))$ which is
 $A_\infty$-quasi-isomorphic to $U(\g)$. Moreover, one can arrange
 it so that the inclusion
$H(\g)\longrightarrow S(H(\g))^\circ$
is a strict map of $L_\infty$-algebras.

We have not explained yet how to
define universal envelopes of
$L_\infty$-algebras, however; we do it only for minimal algebras. In this case, we can take the dg Lie
algebra $\mathcal L\mathcal C(\g)$ that comes equipped with an $L_\infty$-quasi-isomorphism $\mathcal L\mathcal C(\g)\longrightarrow \g$ ---this is the
so-called rectification theorem, 
obtained by the bar-cobar construction--- so we can proceed
with the prescription of Moreno-Fern\'andez to define on $S(\g)$ an
$A_\infty$-structure so that the
PBW inclusion $\g\longrightarrow S(\g)^\circ$ is a strict map of $L_\infty$-algebras. We call this the \new{Moreno-Fern\'andez universal enveloping algebra of $\g$.}

\bigskip

\subsection{The derived PBW property of the morphism \texorpdfstring{$\hLie\longrightarrow 
\hAss$}{LieAss}}

\begin{theorem}\label{thm:main} The morphism $f:\hLie\longrightarrow 
\hAss$ is almost-free, so it is derived PBW.
\end{theorem}

\begin{proof} 
We begin with recalling from~\cite{LV}*{Prop.~9.1.5} that the Loday--Livernet
presentation of the associative operad is given by a commutative non-associative product
$x_1x_2$ and an anti-symmetric Lie bracket $[x_1,x_2]$ that is a derivation of the product
and satisfies the identity
\[ (x_1x_2)x_3-x_1(x_2x_3) = [x_2,[x_1,x_3]]. \]
Let us consider the weight grading on the space of generators which assigns weight zero to the Lie bracket and 
weight one to the product. The associated graded relations with respect to this filtration are the relations of the Poisson operad, and the underlying $\Sigma$-modules of $\Pois$ and $\Ass$ are isomorphic. We are therefore in the situation where result of Appendix applies: there exists a quasi-free resolution of $\Ass$ whose differential is obtained from $d_{\hPois}$ by a perturbation that lowers the weight grading. Moreover, the space of generators of this resolution can be identified with the Koszul dual cooperad of $\Pois$ whose underlying $\Sigma$-module is isomorphic to that of the Koszul dual cooperad of $\Ass$; therefore, this resolution has to be minimal and isomorphic to $\hAss$. 
The dg suboperad $\Lie_\infty$ is in weight filtration zero, so the weight filtration is a filtration of right $\hLie$-modules whose associated graded operad is $\hPois$. 

To complete the proof, we will show that there is a chain homotopy equivalence of right $\,\hLie$-modules $\pi :\hPois\longrightarrow \Com\circ \hLie$. For this, we use the language of distributive laws between operads~\cite{LV}*{Sec.~8.6.3}. The distributive law $\lambda$ that gives rise to the isomorphism $\Pois = \Com\vee_\lambda \Lie$ can be enhanced to 
a distributive law $\lambda'$ between the operads $\Com$ and $\hLie$, for which all higher brackets are derivations with
respect to the commutative product; there is a surjective quasi-isomorphism $\Com\vee_{\lambda'}\hLie
\rightarrow \Com \vee_\lambda \Lie = \Pois$. Hence, we get a surjective quasi-isomorphism 
 \[
\hPois\longrightarrow \Com \circ \hLie = \Com\vee_{\lambda'}\hLie.
 \]
Since we are working over a field of characteristic zero, we can produce a map $i : \mathsf{Com}\longrightarrow \mathsf{Com}_\infty$ such that $p i =1$, where $p$ is the projection onto homology. This then gives us a
map $j:\mathsf{Com}\circ \hLie \longrightarrow
\hPois$ by composing with the composition 
of $\hPois$, and this map is a section of
$\pi$. Since $\hPois$ is free as an operad,
and since we're working over a field of characteristic zero, we can produce an equivariant
contracting homotopy $h$ for $j\pi$. This 
completes the proof that $\hLie\longrightarrow
\hAss$ is almost-free and, by Theorem 1.1,
it is derived PBW. 
\end{proof}

The algebras over the operad $\Com\vee_{\lambda'}\hLie$ used in the proof are sometimes
called \emph{homotopy Poisson algebras} or \emph{$P_\infty$-algebras} in the literature,
even though this operad is not cofibrant. These have been considered by A. S. Cattaneo and 
G. Felder, and independently by T. Voronov, and are related to the theory of
Lie and Courant algebroids, and Poisson manifolds, see~\cites{Voronov, Voronov2, Catt} for example. 

\subsection{Quillen theorem for \texorpdfstring{$L_\infty$}{Linfty}-algebras}

We recall that if $\g$ is an $L_\infty$-algebra,
the \new{bar construction on $\g$} is the
commutative dgc coalgebra $\mathcal{C}(\g)$
with underlying coalgebra $S^c(s\g)$, the
free commutative coalgebra on the suspension
of $\g$, and with differential $d:\mathcal{C}(\g)\longrightarrow \mathcal{C}(\g)$ induced
from the higher brackets of $\g$; the higher
Jacobi identities for these higher brackets
are equivalent to the single equation $d^2=0$.
We write $x_1\wedge\cdots\wedge x_t$ a
generic element from $\mathcal{C}(\g)$, omitting
the suspensions signs for ease of notation. 
Observe this element is simply the anti-symmetrization of the corresponding elementary tensor $sx_1\otimes\cdots \otimes sx_t$ in $T^c(s\g)$.

Similarly, if $A$ is an $A_\infty$-algebra,
the \new{bar construction on $A$} is the
dgc coalgebra $BA$ with underlying coalgebra
$T^c(sA)$, the free coalgebra on the suspension
of $A$, and with differential $d: BA\longrightarrow BA$ induced from the higher
products of $A$; the Stasheff identitis for
these higher products are equivalnt to the 
single equation $d^2$. We write a $[x_1\vert\cdots\vert x_t]$ a generic element
of $BA$. Observe that we can
apply both constructions, in particular,
to dg Lie algebras and dga algebras, of course.
Finally, if $C$ is a dgc coalgebra, the 
\new{cobar constuction on $C$} is the dga algebra
$\Omega C$ with underlying algebra $T(s^{-1}C)$,
the free algebra on the desuspension of $C$,
with differential induced from the comultiplication and differential of $C$.
Concretely, it is the unique derivation
of $\Omega C$ that extends the map $s^{-1}C\longrightarrow \Omega C$ such that
$d(s^{-1}c) = s^{-1}dc -s^{-1}\otimes s^{-1}\Delta c$.

\smallskip
Let us recall from~\cite{TFH}
that for a dg-Lie algebra $(\g,d)$ there
is a quasi-isomorphism $q:\mathcal C(\g) \longrightarrow 
BU(\g)$ where the left hand side is the Quillen
construction on $\g$ (that coincides with the Chevalley-Eilenberg complex of $\g$) and the right hand side
is the associative bar construction on the universal envelope of $\g$.  We reminder the
reader that $q$ is determined uniquely
by a map $\tau : \mathcal{C}(\g) \longrightarrow U(\g)$, which we call the \new{twisting cochain
associated to $q$}. That this be a twisting
cochain is equivalent to the Maurer--Cartan
equation $d\tau +\tau\star \tau = 0$. Here, 
the
star product of
the convolution dga algebra $A=\hom(\mathcal{C}(\g),U(\g))$ is defined by $\star = \mu(-\otimes -)\Delta$. 
\smallskip

We also remind the reader that $\tau$ simply sends a 
generator $sg\in s\g$ to 
the class of its desuspension in $U(\g)$.
The following theorem extends this picture
to the case of $L_\infty$-algebras ---we will observe 
below $\tau$ \emph{also} defines a twisting morphism in this
more general setting. In this case, the 
Maurer--Cartan equation is replaced by a higher
analog: since $U(\g)$ is now an 
$A_\infty$-algebra, the convolution algebra
above is in fact again an $A_\infty$-algebra,
so that for each $t\in\NN$
and each $f_1,\ldots,f_t\in A$, we have
higher multiplications defined by 
 $m_t(f_1,\cdots,f_t) = \mu_t (f_1\otimes\cdots \otimes f_t)\Delta^{(t)}$. The Maurer--Cartan
 equation incorporates these higher products
 and now reads:
\[ d\tau + \sum_{t\geqslant 2} \mu_t(\tau,\ldots,\tau) = 0.\]
It is useful to remember this conditions simply
codifies, with the least amount of information
possible, the fact that the corresponding map
$\mathcal{C}(\g)\longrightarrow BU(\g)$ is one
of dgc coalgebras, and that this assignment
defines a bijection between dgc coalgebra maps
$\mathcal{C}(\g)\longrightarrow BU(\g)$ and
twisting cochains $\mathcal{C}(\g)\longrightarrow 
U(\g)$. For details, we refer the reader to the
book~\cite{LV}.
\pagebreak

Before stating the result, we recall 
that the category of dgc coalgebras $\Cog$ 
is a model
category where the cofibrations are the
degree-wise monomorphisms, the weak equivalences
are the quasi-isomorphisms and cofibrations
satisfy the right lifting property with 
respect to acyclic fibrations; all objects are cofibrant, and the 
fibrant dgc coalgebras are the quasi-free ones, see~\cite{Val}. A map of
coalgebras is a weak equivalence if and only if
its image under the cobar construction is a 
quasi-isomorphism of dga algebras. This class is
\emph{strictly contained} in the class of quasi-isomorphisms.

\begin{theorem}\label{thm:QM}
For any 
$L_\infty$-algebra $\g$ the map $q : \mathcal C(\g) \longrightarrow 
BU(\g)$ corresponding to the twisting cochain
$\mathcal C(\g) \longrightarrow U(\g)$ that assigns
a generator $sg\in \mathcal{C}(\g)$ the
class of its desuspension is an acyclic cofibration of coalgebras.
If $\g$ is minimal, then there are acyclic cofibrations of the form $\mathcal{C}(\g) \longrightarrow BS(\g)$. 
\end{theorem}

\begin{proof}
Let us begin by observing that the Maurer--Cartan 
equation for $\tau$ is simply a restatement that 
the higher products in $U(\g)$ antisymmetrize to 
the higher brackets of $\g$. Indeed, for
a generator $x = x_1\wedge\cdots\wedge x_t$ in the domain, we have that
$\tau dx$ is equal to the higher
bracket $[x_1,\cdots,x_t]$,
up to signs, 
while the only non-zero term involving
higher products of $\tau$
involves $\Delta^{(t)}(x)$. This
is just the signed sum over $\sigma\in S_t$ of $\sigma x$, and then
$m_t(\tau,\ldots,\tau)(x)$ is precisely
the antisymmetrized higher
product of $x_1\otimes\cdots\otimes x_t$. 
\smallskip

To see that $q$ is a
weak equivalence one can show that $\Omega\mathcal{C}(\g)\longrightarrow S(\g)$ is one. To do this, one can argue 
as in~\cite{Bar}*{Theorem 3}, or note that
 there is a morphism of $A_\infty$-algebras
$BU(\g)\longrightarrow BS(\g)$, where the right
hand side is Baranovsky's universal envelope, that
is a quasi-isomorphism. Indeed, the map $H(U(\g))\longrightarrow S(\g)$ is an automorphism of
the enveloping associative algebra of $(\g,l_2)$. This implies 
that $\Omega q$ is a quasi-isomorphism, since
$\varepsilon_{U(\g)}:\Omega BU(\g)\longrightarrow U(\g)$ is
one, and $\Omega \mathcal C(\g)\longrightarrow
 U(\g)$ is a quasi-isomorphism that factors as  $\varepsilon_{U(\g)}\Omega q$.
\end{proof}

We write $A^\antishriek$ for the homology of $BA$
and, for an $L_\infty$-algebra $\g$, we write $\g^\antishriek$ for the 
homology of $\mathcal{C}(\g)$. An immediate corollary of the 
previous theorem is the following; all claims follow from the general theory
of twisting cochains; see~\cite{Bar}*{Theorem 2}. 

\begin{cor}
Let $\g$ be a minimal $L_\infty$-algebra. Then
\begin{tenumerate}
\item the twisted complex $\mathcal{C}(\g)\otimes_\tau S(\g)$ is
quasi-isomorphic to $\kk[0]$,
\item we have an
isomorphism  $\g^\antishriek \longrightarrow S(\g)^!$ of
minimal $A_\infty$-coalgebras,
\item the categories of $\g$-modules and $U(\g)$-modules are
equivalent and,
\item the functors $\;\?\otimes_\tau
S(\g): \mathsf{D}(\mathcal{C}(\g))\leftrightarrows \mathsf{D}(S(\g)):\;\?\otimes_\tau
\mathcal{C}(\g)$ are mutually\\ inverse derived equivalences.\hfill \qed
\end{tenumerate}
\end{cor}

\subsection{The universal envelope as a functor on the homotopy category}

Suppose that $f :\mathcal{C}(\g)\longrightarrow\mathcal{C}(\g')$ is an $L_\infty$-morphism. The
coalgebra 
$BU(\g)$ is fibrant and the Quillen map of
$\g'$ is
an acyclic cofibration, so we can produce a lift
$U(f): BU(\g)\longrightarrow BU(\g')$. The following lemma implies this assignment has
good homotopical properties; in particular,
it is well defined up to homotopy. 

\begin{lemma}
Let $\varphi,\varphi' : BU(\g)\longrightarrow BU(\g')$ be such that $\varphi q = \varphi' q=q'f$.  Then $\varphi\simeq\varphi'$ as maps of 
dgc coalgebras. In particular, for a second
map $g :\mathcal{C}(\g')\longrightarrow\mathcal{C}(\g'')$, we have that $U(g)U(f) \simeq U(gf)$
for any choice of lifts of $f$, $g$ and $gf$. 
\end{lemma}

\begin{proof}
The reader can consult~\cite{Hov}*{Sections 1.1-1.2}
for details on the elements of model categories
used in this proof. We remind the reader $\Cog$
is the model category of coalgebras where the
cofibrations are the injections and the
weak equivalences are created by the cobar
functor.

The map $q$ is a weak equivalence between
cofibrant objects and $BU(\g')$ is fibrant. 
It follows from Lemma 1.1.12 that we have an induced isomorphism
\[
 q^*: \Cog(BU(\g),BU(\g'))/\simeq_r 
 	\longrightarrow 
 	\Cog(\mathcal{C}(\g),BU(\g'))/\simeq_r\] Since all objects in $\Cog$ are
 	cofibrant and $B(-)$ has image in fibrant objects, we deduce from Proposition 1.2.5 (v)
 	that we can replace the right homotopy relation
 	$\simeq_r$ unambiguously by the homotopy relation $\simeq$, and finally  Theorem 1.2.10 (ii) implies that
 	this isomorphism identifies naturally
 	with an isomorphism 
 	\[
 q^*: [BU(\g),BU(\g')]
 	\longrightarrow 
 	[\mathcal{C}(\g),BU(\g')]\]
 	which implies that if $q^*(\varphi) =q^*(\varphi')$ then these two maps must be homotopic.\end{proof}

With this at hand we have the following result.

\begin{theorem}
The universal envelope preserves weak equivalences
of $\,L_\infty$-algebras, so it descends to a functor
on the homotopy category. Moreover, if for a minimal $L_\infty$-algebra $\g$ we identify $S(\g)$ with
the universal envelope of the Lie algebra
$(\g,l_2)$, then
\begin{tenumerate}
\item there exists
 a minimal $A_\infty$-algebra structure on $S(\g)$
for which the inclusion $\g\longrightarrow S(\g)^\circ$ is
a strict map of $L_\infty$-algebras and, moreover,
\item any such $A_\infty$-structure is $A_\infty$-isomorphic to this
one, so that the universal envelopes of Baranovsky and Moreno-Fernández
are $A_\infty$-isomorphic and $A_\infty$-quasi-isomorphic to $U(\g)$.
\end{tenumerate}
\end{theorem}

\begin{proof}
The claim about weak equivalences is immediate since the Quillen
map is a weak equivalence. The first claim follows from the homotopy
transfer theorem, and the second one follows
from the universal property of the enveloping
algebra and our main result. Indeed, suppose that $S(\g)$
is endowed with an $A_\infty$-algebra structure as in the statement of the theorem. The strict map of $L_\infty$-algebras $\g\longrightarrow S(\g)$ gives us an $A_\infty$-quasi-isomorphism  
$BU(\g) \longrightarrow BS(\g)$, for
the map $U(\g)\longrightarrow S(\g)$ induces an
isomorphism of algebras: we have shown that $H(U(\g))$ is naturally identified  with the universal 
envelope of the Lie algebra $(\g,b_2)$, and so
does $S(\g)$ in Baranosvky's construction.
\end{proof}

\noindent 
Let us recall that one can rectify every
$L_\infty$-algebra $\g$ to a bona-fide 
dg Lie algebra. In this way, one can show that
the homotopy category of dg Lie algebras
$\chLie$ and the
category of $L_\infty$-algebras up to 
quasi-isomorphism are equivalent. An
analogous result holds for dga algebras. Hence, we can view the universal envelope
functor ${}_{\hLie}\mathsf{Alg}\longrightarrow
{}_{\hAss}\mathsf{Alg}$ as a choice for a 
representative of the functor that assigns to a dg Lie 
algebra its universal envelope.
A restatement of our results is the following.

\begin{cor}
The constructions of Baranovsky, 
Lada--Markl
and Moreno-Fernández give representatives
for the universal envelope functor 
$\,\mathsf{ho}(U):\chLie \longrightarrow \chAss$.
Moreover, if $F : {}_{\hLie}\mathsf{Alg}\longrightarrow
{}_{\hAss}\mathsf{Alg}$ 
satisfies the conditions:
\begin{tenumerate}
\item the $A_\infty$-algebra $F(\g)$ is minimal whenever the $L_\infty$ algebra $\g$ is minimal, 
\item the underlying vector space to $F(\g)$ is the symmetric algebra $S(\g)$ and,
\item the higher products of
$F(\g)$ induce the higher brackets on $\g\subseteq S(\g)$,
\end{tenumerate} 
then $F$ descends to the homotopy category
and $\mathsf{ho}(F) = \mathsf{ho}(U)$.  
 \qed
\end{cor}

\subsection{Cohomology groups}

Recall that if $P$ is an operad and $A$ is a
$P$-algebra, operadic cohomology of $A$ is,
by definition, the cohomology of the complex
of $P$-derivations $\Der(B,A)$ where $B$
is a cofibrant resolution of $A$ in the model
category of $P$-algebras. This complex is
quasi-isomorphic to the dg Lie algebra $\Der(B)$
of derivations of $B$ to itself.
 More generally, the
cohomology of $A$ with values in an operadic
$A$-module $M$ is, by definition, the cohomology
of the complex of $P$-derivations $\Der(B,M)$
where $M$ is given a $B$-module structure 
through the map $B\longrightarrow A$. We
refer the reader to~\cite{LV,Fresse} for
details.

\pagebreak

Let us fix a minimal $L_\infty$-algebra $\g$, 
a minimal $A_\infty$-algebra $S(\g)$ that models 
$U(\g)$, and an acyclic twisting cochain 
$\mathcal{C}(\g)\longrightarrow S(\g)$. Observe 
that then $\Omega \mathcal{C}(\g)$ is a 
quasi-free model and hence a cofibrant replacement for $S(\g)$,
so that the operadic cohomology $H^*_{\mathsf{Ass}}(S(\g))$ can be computed through the complex of
derivations of the dg-algebra $\Omega\mathcal{C}(\g)$ with values in $S(\g)$. This 
receives
a map from the complex of derivations $\Der(\Omega\mathcal{C}(\g),\g)$ through
the strict map of $L_\infty$-algebras $\g\longrightarrow S(\g)$. We then obtain the following
result, which is expected.

\begin{theorem}
The maps above induce an isomorphism $H^*_{\mathsf{Ass}}(S(\g))
\longrightarrow H^*_{\mathsf{Lie}}(\g,S(\g))$ and
an injection $H^*_{\mathsf{Lie}}(\g)
\longrightarrow H^*_{\mathsf{Lie}}(\g,S(\g))$ in
operadic cohomology groups. \qed
\end{theorem}

\subsection{Derived PBW theorems for associative envelopes}

To each symmetric operad $P$ one can assign an 
\new{associative universal enveloping functor} 
$U_P : {}_P\Alg\longrightarrow
\Alg$ from the category of  $P${-algebras} to the category of associative algebras, see~\cite{GK} and \cite{Khor::UP}*{Section 1} for details. The associative algebra $U_P(V)$ associated with a $P$-algebra $V$ satisfies the universal property that the category of left $U_P(V)$-modules is equivalent to the category of left modules over a $P$-algebra $V$. The first
author explains in detail in~\cite{Khor::UP}*{Section 1.2} that one can also interpret this
associative universal envelope as a left adjoint
to a restriction functor, so that one may study it
using the formalism of~\cite{PBW}.

Indeed, the universal enveloping functor $U_P$ is obtained from the left adjoint $f_{!}$ corresponding to the restriction functor for the map of colored operads
$f:(P,\Bbbk) \to (P,\partial P)$.
Here $(P,\Bbbk)$ is the two-colored operad 
that governs pairs of the form $(V,M)$ where $V$ is a $P$-algebra and $M$ is a vector space,
while $(P,\partial P)$ is the two-colored
operad that governs pairs of the form $(V,M)$
where $V$ is a $P$-algebra and $M$ is a left $V$-module. Recall that $\partial P$ is the \new{derivative}
of the symmetric sequence $P$, and it is
obtained from $P$ by adding an extra color 
to the output and one of the inputs of each 
operation in $P$. With this at hand, one can
check that $f_!(V,\Bbbk) = (V, U_P(V))$. 

The fact that $U_P(V)$ is an associative
algebra (and not merely a vector space) 
comes from an additional structure
on $\partial P$: although this collection is no longer an
operad in an obvious way, it is a \new{twisted associative
algebra} ---that is, an associative algebra
for the Cauchy product in symmetric sequences---
where the product is given by grating
the root on the unique new colored input. In this
way $U_P(V)$ inherits an associative
algebra structure, and one can check that the
datum of a left $U_P(V)$-modules is the same 
as that of a left $V$-module. 

The first author has shown in~\cite{Khor::UP} that the functor $U_P(-)$ satisfies the PBW property whenever $P$ admits a Gr\"obner basis whose leading monomials are given by left combs.
Recall that a shuffle monomial is called a left comb if and only if all its inner vertices belong to the leftmost branch of a shuffle tree, that is, the path connecting the first input and the output.
Another criterion was proposed in~\cite{Khor::UP} 
in case $P$ is a \emph{Koszul operad}: 
\[\claim{the functor $U_P$ satisfies the PBW property if and only if the twisted associative algebra $\partial P^{!}$ is quadratic, Koszul and generated by $\partial X$ as a twisted associative algebra.} \]
Here $X$ is the generating symmetric sequence of 
the symmetric operad $P$. We claim that the same conditions are sufficient in the derived setting;
we still assume that $P$ is Koszul.

\begin{theorem}
\label{thm::UAss}
If the associative universal envelope $U_P$ satisfies the PBW property then the corresponding derived associative universal envelope $U_{P_\infty}$ satisfies the derived PBW property.
In particular, if an operad $P$ admits a quadratic Gr\"obner basis whose leading monomials are given by left combs then the universal enveloping functor $U_{P_\infty}$ satisfies derived PBW.
\end{theorem}
\begin{proof}
	Let $P_{\infty}:=\Omega P^{\antishriek}$ be the minimal cofibrant model of $P$ generated by the Koszul-dual cooperad $P^{\antishriek}$ and denote by $(\partial P^{\antishriek})^{\antishriek}_{\TwAs}$ the twisted associative algebra that is Koszul dual to the twisted associative algebra $\partial P^{\antishriek}$. 
	One of the main observations in~\cite{Khor::UP} is the commutation of the coloring procedure $P\mapsto \partial P$ and cobar constructions. So
that there is a natural isomorphism $\partial (\Omega P^{\antishriek}) \longrightarrow \Omega(\partial P^{\antishriek})$.
	Each element of a (colored) cobar construction of a (colored) (co)operad is represented by a (colored) operadic tree $T$. 
	The colored cobar construction of the colored operad $\partial P$ admits the PBW-filtration given by the number of edges connecting the branch colored in a new color  and the remaining part of an operadic tree $T$. The associated graded complex is quasi-isomorphic to the composition
	\[\Omega_{\TwAs}(\partial P^{\antishriek})\circ \Omega P^{\antishriek} \simeq 
	\Omega_{\TwAs}(\partial P^{\antishriek})\circ P_{\infty}.\]
	As shown in~\cite{Khor::UP} the Koszulness of the twisted associative algebra $\partial P^{!}$ is a necessary condition for $U_P$ to satisfy the ordinary PBW criterion. Hence the PBW property for $U_P$ implies the existence of a chain equivalence $s:  (\partial P^{\antishriek})^{\antishriek}_{\TwAs} \rightarrow   \Omega_{\TwAs}(\partial P^{\antishriek})$, which shows the map $\partial P_{\infty}$ is almost-free, and finishes our proof of the theorem.
\end{proof}

An interesting consequence of the above 
``PBW-rigidity'' phenomenon for associative universal 
envelopes is the following result, which shows
that to compute the derived associative
universal envelope of a $P$-algebra, one may
only resolve only one variable in the functor
$U_P(-)$. 

\begin{cor}
Let $P$ be Koszul. Given a $P_{\infty}$-algebra $A$ and a dg $P$-model 
$B$ of $A$, the corresponding associative universal 
envelopes are quasi-isomorphic, namely, there is 
always a quasi-isomorphism $U_P(B) 
{\longrightarrow} U_{P_\infty}(A)$ of dga algebras.
\qed\end{cor}

\section{Further directions}

\subsection{Derived universal envelopes}

Let us fix a map $f:P\longrightarrow Q$ of
dg operads over a field $\kk$. 
Although we focused on $P$-algebras ---left $P$-modules concentrated in arity $0$---, the
universal envelope defines a map
$f_!:{}_P\Mod \longrightarrow \Ch_\kk$, which
is given explicitly by $f_!(X) = Q\circ_P X$. 
According to 
\cite{Fresse}, the category $\Mod_P$ of right 
$P$-modules admits a cofibrantly 
generated model structure where fibrations
and weak-equivalences are defined point-wise.
In particular, we can consider a cofibrant
replacement $Q^*$ of $Q$, for example, the
two sided bar construction $B(Q,P,P)$, and define 
$\mathbb Lf_!(X) = Q^* \circ_P X$, which
gives us the object $\Tor^P(Q,X)$. Following
the procedure of~\cite{Guide}*{Chapter 7}, we
can produce a Eilenberg--Moore type spectral sequence which gives a fine tool to study derived PBW
phenomena; our arguments essentially consider
the situation when there is an immediate collapse
of this sequence due to $Q$ begin almost-free.
It would be interesting to consider situations
where certain restrictions on $Q$, other than
almost-freeness, allow us to obtain derived 
PBW theorems. 

\subsection{Duflo-type results for higher centres}

Although the canonical map $\alpha:S(\g)\longrightarrow
U(\g)$ is not a map of algebras (since the
source is commutative, but the target is not),
we can consider the adjoint action of $\g$
on both spaces. It is well known that the
map above is then one of $\g$-modules and thus
induces a map $\alpha^\g:S(\g)^\g\longrightarrow U(\g)^\g$, where $U(\g)^\g$ is just the center
of $U(\g)$, a commutative algebra. This map is,
however, not an isomorphism of algebras either. A remarkable result of M. Duflo shows that one can construct from this an isomorphism of
algebras 
\[ \alpha^\g\circ J^{1/2}:S(\g)^\g\longrightarrow U(\g)^\g,\]
known as the \new{Duflo isomorphism},
through a suitable (and quite involved) modification of this map through an automorphism $J$ of $S(\g)$. In fact, M. Pevzner and C. Torossian 
proved that the Duflo isomorphism is part of
an isomorphism of Lie cohomology groups 
\[ H^*(\g,S(\g))\longrightarrow H^*(\g,U(\g))\] 
induced from a quasi-isomorphism $C^*(\g,S(\g))
\longrightarrow C^*(\g,U(\g))$ between the
corresponding Chevalley--Eilenberg complexes,
following insight of 
M. Kontsevich.
We refer the reader to~\cite{Duflo} for details
and useful references.
In~\cite{VinBen} the authors define the $\infty$-centre of a minimal $A_\infty$-algebra, which
can be used, for example, to describe the image 
of the ``wrong way'' map 
\[ H_{*+d}(LX)\longrightarrow H_*(\Omega X)\] onto the Pontryagin algebra of a simply connected smooth
oriented $d$-manifold $X$: the image of this 
map is precisely the $\infty$-centre of 
$H_*(\Omega X)$. In particular we can consider, for any minimal
$L_\infty$-algebra, the $\infty$-centre $Z_\infty
S(\g)$, which is a commutative algebra. It would be interesting to understand
this higher centre and explore the possiblity
of extending the results of Duflo and 
Pevzner--Torossian  to this setting.

\appendix 
  \section[Models of operads via homological perturbation]{Models of operads via homological perturbation\footnote{by Vladimir Dotsenko}}
\label{ssec:perturbation}

This section records an instance of a general homological perturbation argument which allows one to obtain, in a range of cases, a resolution of a filtered object from the one of its associated graded object. A similar argument for a perturbative construction of a resolution of a shuffle operad with a Gr\"obner basis is featured in \cite{DotKho}*{Th.~4.1}. We keep the assumption on the characteristic of the ground field.

 
Let us consider a (non-dg) symmetric operad $Q=\mathcal{T}(X)/(R)$ generated by a finite dimensional $\Sigma$-module 
$X$ concentrated in arities greater than one, subject to a finite dimensional space of relations $R$. Suppose that the $\Sigma$-module $X$ is equipped with a non-negative weight grading, 
\[X=\bigoplus_{n\ge 0}X_{(n)}.\] This weight grading gives rise to a weight grading of the free operad $\mathcal{T}(X)$, and hence an increasing filtration $F^\bullet \mathcal{T}(X)$ such that $F^k\mathcal{T}(X)$ is spanned by all elements of weight at most~$k$. This filtration gives rise to a filtration on each $\Sigma$-submodule of $\mathcal{T}(X)$. In particular, we may consider the operad $P=\mathcal{T}(X) /(\gr_F R)$. In general, there is an isomorphism of $\Sigma$-modules between $Q$ and $\gr_F Q=\gr_F\mathcal{T}(X) / \gr_F(R)$, and a surjection of $\Sigma$-modules from $P$ onto $Q$.

\begin{theorem*}
Suppose that the underlying $\Sigma$-modules of the operads $P$ and $Q$ are isomorphic. Consider the minimal quasi-free resolution $P_\infty\longrightarrow P$ in the category of weight graded operads. There exists a quasi-free resolution $Q_\infty\longrightarrow Q$ whose underlying free operad is the same and the differential is obtained from the differential of $P_\infty$ by a perturbation that lowers the weight grading. 
\end{theorem*}

\begin{proof} 
Assume that $P_\infty$ is of the form $(\mathcal{T}(W),d)$, where $W$ is an $\Sigma$-module that is bi-graded, by weight and by homological degree. We shall prove that there exists a quasi-free resolution $Q_\infty=(\mathcal{T}(W),d+d')$ of $Q$ so that $d'$ is strictly weight decreasing.

Since the resolution $P_\infty=(\mathcal{T}(W),d)$ is minimal, it follows in particular that $W_0= X$, $W_1 =\kk s\otimes\gr_F R$. The operad $\mathcal{T}(W)$ can be mapped to both the operad $P$ and the operad $Q$: one may project it onto its part of homological degree $0$, the latter is isomorphic to the free operad $\mathcal{T}(X)$ which admits obvious projection maps to $P$ and to $Q$. Let us choose splittings for those projections; this amounts to exhibiting two idempotent endomorphisms $\bar{\pi}$ and $\pi$ of $\mathcal{T}(W)$ such that both of which annihilate all elements of positive homological degree, and such that the former annihilates the ideal $(\gr_F R)\subset \mathcal{T}(X)=\mathcal{T}(W)_0$ and the latter annihilates the $(R)\subset \mathcal{T}(X)=\mathcal{T}(W)_0$.

Since $P$ is finitely generated and has no generators of arity $1$, components of the free operad $\mathcal{T}(W)$ are finite-dimensional, and there exists a weight graded homotopy $h\colon \mathcal{T}(W) \longrightarrow \mathcal{T}(W)$ such that $h^2=0$ and $[d,h] = 1 - \bar{\pi}$.

We are going to define a derivation $D\colon \mathcal{T}(W) \longrightarrow \mathcal{T}(W)$ of degree $-1$ and a contracting homotopy $H\colon \ker(D) \to \mathcal{T}(W)$ of degree $+1$. Note that a derivation is fully determined by the images of generators, and that $D|_{W_0}=0$ since $\mathcal{T}(W)$ has no elements of negative homological degree. For each element $x$ of $\mathcal{T}(W)$ of a certain homological degree, we call the ``leading term'' of $x$ the homogeneous part of $x$ of maximal possible weight grading; we denote it by $\widehat{x}$. 

We shall prove by induction on $k$ that one can define the values of $D$ on generators of homological degree $k+1$ and the values of $H$ on elements of $\ker(D)$ of homological degree $k$ so that the following five conditions hold:
\begin{tenumerate}
\item for all elements $x\in\mathcal{T}(W)$, the leading term of the difference $D(x)-d(x)$ is of weight lower than that of $x$,
\item we have $D^2=0$ on generators of homological degree $k+1$,
\item the leading term of the difference $H(x)-h(\widehat{x})$ is of weight lower than that of $x$,
\item we have $DH=1-\pi$ on elements of $\ker(D)$ of homological degree $k$.
\end{tenumerate}

As a basis of induction, we shall choose a basis of $W_1 = \kk s\otimes\gr_F R$, and set  
$D(s\otimes r')=r$
where $r$ is some element of $R$ for which $\widehat{r}=r'$. We note that $D(s r')-d(s r')=r-r'$ has smaller weight than $r'$, so Condition~(1) is satisfied. Condition~(2) is satisfied for degree reasons, as there are no elements of negative homological degree. Condition~(1) together with the fact that $\bar{\pi}=\pi$ on elements of weight zero implies that the leading term of $Dh(\widehat{x})$ is
 \[
dh(\widehat{x})=(1-\bar{\pi})(\widehat{x})=(1-\pi)(\widehat{x})=\widehat{x}-\pi(\widehat{x}),
 \]
and so we may define $H$ on elements of homological degree zero by induction on weight as follows. On elements $x$ of weight zero,  
we put $H(x)=h(x)$, and on elements $x$ of positive weight, we put
 \[
H(x)=h(\widehat{x})+H(x-\pi(x)-Dh(\widehat{x})).
 \]
Both Condition~(3) and Condition~(4) are proved by induction on weight. For former one, the inductive argument is almost trivial; we shall show how to prove the latter. On elements of weight zero, we have $H(x)=h(x)$ and $\bar{\pi}=\pi$, so Condition~(4) is true:
 \[
DH(x)=dh(x)=[d,h](x)=(1-\bar{\pi})(x)=(1-\pi)(x).
 \]
For elements of positive weight, we have, by induction, 
\begin{align*}
DH(x) &=Dh(\widehat{x})+DH(x-\pi(x)-Dh(\widehat{x})) \\
&=Dh(\widehat{x})+(1-\pi)((1-\pi)x-Dh(\widehat{x}))\\
&= (1-\pi)^2(x)+\pi(Dh(\widehat{x}))\\
&=(1-\pi)(x),
\end{align*} 
since $\pi$ vanishes on the image of $D=(R)$ and $1-\pi$ is a projector. To carry the inductive step, we proceed in a similar way. To define the image under $D$ of a generator of homological degree $k+1>1$, we put $D(x)=d(x)-HDd(x)$. 
Condition~(1) now easily follows by induction. For Condition~(2), we note that 
\begin{align*}
D^2(x)&=D(d(x)-HDd(x)) \\
	&=Dd(x)-DH(Dd(x))\\
    &=Dd(x)-(1-\pi)Dd(x)=\pi(Dd(x))=0,
\end{align*}
since $Dd(x)\in\ker(D)$, and $\pi$ vanishes on the image of $D$. From that, we see that whenever $x\in\ker(D)$, we have $x-Dh(\widehat{x})\in\ker(D)$. Using Condition~(1) and the fact that $\bar{\pi}$ vanishes on elements of positive homological degree, we see that the leading term of $Dh(\widehat{x})$ is $dh(\widehat{x})=(1-\bar{\pi})(\widehat{x})=\widehat{x}$, so the leading term of $x-Dh(\widehat{x})$ is of weight lower than that of $x$. Consequently, we may define $H$ on elements of $\ker(D)$ of homological degree $k>0$ by the same inductive argument: on elements $x$ of weight zero, we put $H(x)=h(x)$, and on elements $x$ of positive weight, we put
 \[
H(x)=h(\widehat{x})+H(x-Dh(\widehat{x})).
 \]
Once again, a simple inductive argument shows that Conditions~(3) and~(4) are satisfied, which
completes the construction of $D$ and $H$.  

We conclude that $D$ makes $\mathcal{T}(W)$ a dg operad, that the homology of that operad is isomorphic to~$Q$, and that differential $D$ is obtained from $d$ by a perturbation $d'$ that lowers the weight grading, as required.
\end{proof}

\Addresses

\end{document}